\documentclass[12pt,a4paper]{amsart}
\usepackage{amsmath}
\usepackage[left=2.3cm,right=2.3cm,top=4cm,bottom=4cm]{geometry} 
\usepackage{meintex}
\usepackage[font={small,it}]{caption}

\numberwithin{equation}{section}

	\usepackage{amsmath, amssymb, amsthm}
\usepackage[utf8]{inputenc}
\usepackage{geometry}
\usepackage{bbm}
\usepackage{todonotes}
\usepackage{tikz}
\usepackage{enumitem}
\usepackage{hyperref}
\usepackage{csvsimple}
\usepackage{float}
\usepackage{multicol}

\newtheorem{lemma}{Lemma}
\newtheorem{remark}{Remark}
\renewcommand{\Xi}{X^i(t)}
\newcommand{\Xit}{X^i_t}
\newcommand{\pit}{p_f^i}
\newcommand{\vf}{v_f}

\newcommand{\Hxp}{\text{H}(f(X^i(t)) - f(\pit))}
\newcommand{\Hvp}{\text{H}(f(\vf) - f(\pit))}
\newcommand{\Hxv}{\text{H}(f(X^i(t)) - f(\vf))}
\newcommand{\Hpv}{\text{H}(f(\pit)-f(\vf))}

\newcommand{\Hexp}{\text{H}^\epsilon(f(X^i(t)) - f(\pit))}
\newcommand{\Hevp}{\text{H}^\epsilon(f(\vf) - f(\pit))}
\newcommand{\Hexv}{\text{H}^\epsilon(f(X^i(t)) - f(\vf))}
\newcommand{\Hepv}{\text{H}^\epsilon(f(\pit)-f(\vf))}

\begin{document}
	
\title[CBO with personal best]{Consensus-Based Global Optimization \newline with Personal Best}

\author[C.~Totzeck]{Claudia Totzeck}
\address[C.~Totzeck]{School of Business Informatics and Mathematics, University of Mannheim, B 6, 68159 Mannheim, Germany}
\email[]{totzeck@uni-mannheim.de}

\author[M.-T.~Wolfram]{Marie-Therese Wolfram}

\address[M.-T.~Wolfram]{ Mathematics Institute, University of Warwick, Gibbet Hill Road, CV47AL Coventry, UK}
\email[]{m.wolfram@warwick.ac.uk}
\address[M.-T.~Wolfram]{RICAM, Altenbergerstr. 69, 4040 Linz, Austria}

\begin{abstract}

In this paper we propose a variant of a consensus-based global optimization (CBO) method that uses personal best information in order to 
compute the global minimum of a non-convex, locally Lipschitz continuous function. The proposed approach is motivated by the original particle swarming algorithms, in which particles 
adjust their position with respect to the personal best, the current global best, and some additive noise. The personal best information along an individual trajectory 
is included with the help of a weighted mean. This weighted mean can be computed very efficiently due to its accumulative structure. It enters the dynamics via an additional drift term. We illustrate the performance with a toy example, analyze the respective 
memory-dependent stochastic system and compare the performance with the original CBO with component-wise noise for several benchmark problems. The proposed method has a higher success rate for computational
	experiments with a small particle number and where the initial particle distribution is disadvantageous with respect to the global minimum.
\end{abstract}

\maketitle

\section{Introduction}

Interacting particle systems play an important role in many applications in science - on the one hand as a modeling framework for social and biological
systems, on the other as a  tool for computational algorithms used in data science. In the latter case the collective behavior of interacting
particle systems is used to solve high-dimensional problems, often resulting from non-convex optimization tasks in data science. Well known
algorithms include particle swarm optimization (PSO) \cite{EberhartKennedy}, ant colony optimization \cite{Dorigo} or evolutionary \cite{Fogel} and genetic algorithms \cite{Holland}.

PSO was first introduced in \cite{EberhartKennedy} and has been successfully used in engineering applications \cite{PoliKennedyBlackwell}. Each particle in a PSO algorithm adjusts its 
position due to information of the global best, personal best and a noise term that allows for exploration of its neighborhood. Consensus-based optimization (CBO) \cite{CBO,CBO2} 
combines the idea of swarm intelligence with consensus formation techniques \cite{ Seneta,Hegselmann, Tadmor} to obtain a global optimization algorithm for non-convex high-dimensional 
problems. On the one hand particles explore the state space via an amplitude modulated random walk. On the other a drift term convects them towards the weighted global best. 
The method was first introduced in \cite{CBO} and analyzed at the mean-field level in \cite{CBO2}. Recent developments of CBO include component-wise diffusion and utilize random mini-batch ideas to reduce the computational cost of calculating the weighted average \cite{Jose}. Other contributions investigate a CBO dynamic that is restricted to the sphere \cite{Fornasier1, Fornasier2}. 
Also, convergence and error estimates for time-discrete consensus-based optimization algorithms have been discussed \cite{HaJinKim}. CBO-type systems are related to large interacting
particle systems, in which the dynamics are driven by weighted average quantities, see \cite{BFOZ2019,CJK2018, WDHK2013}.

The model proposed in the work is based on the component-wise diffusion variant introduced in \cite{Jose} and combines it with personal best information. 
This adjustment is motivated by the original work on PSO by Eberhart and Kennedy \cite{EberhartKennedy}, where the particles move towards a (stochastic) linear combination 
of their personal best or the common global best position. The new information leads to an additional drift term in the dynamics. 
We investigate two types of memory effects - either using a weighted personal best over
time or the personal best value in the past. The latter corresponds to record processes, see \cite{W2013} 
for an overview. The former is used in the presented analysis and approximates the personal best of each particle. 
We expect by arguments similar to the Laplace principle that the weighted mean converges towards the personal best. 

The proposed stochastic dynamics with weighted personal best fall into the class of stochastic functional differential equations. 
These equations are in general non-Markovian and their mean-field limit has been investigated in special cases only. For example, Gadat and Panloup \cite{GP2014} investigated a non-Markovian process with  memory, which corresponds to the weighted average of the drift all along 
the particle's trajectory. This memory term is of a special form allowing them to rewrite the system as a 2-dimensional non-homogeneous Markovian dynamical system. 
Moreover, they exploit this special structure to analyze the existence and long time behavior of solutions as well as the mean-field limit. The strategy of increasing the dimension to get around the non-Markovian nature goes back to the Mori-Zwanzig formalism, see \cite{Z2001}. This
strategy was recently adapted for non-Markovian interacting dynamics by Duong and Pavliotis in \cite{DP2018}. Another interesting work by Kuntzmann, see \cite{K2014}, investigates the ergodic behavior of 
self-interacting diffusions depending on the empirical mean of the process. The proposed generalization of CBO with weighted personal
best does not fall into this category, hence the derivation and analysis of the respective mean-field dynamics, which often give useful insights into the dynamics, is to the
best of the authors' knowledge open.  This applies as well for personal best, where the update of the best function value corresponds to a record process. Hence, we focus 
on the well-posedness of the stochastic system as well as a detailed computational investigation of the dynamics.

This paper is organized as follows: we introduce the particle dynamics with (weighted) personal best in Section \ref{s:modelling} and illustrate 
its dynamics with first toy examples. Section \ref{s:analysis} discusses well-posedness and existence of solutions to the SDE model with weighted personal best. Section \ref{s:numerics}
presents extensive computational experiments of various benchmark optimization problems.

\section{Consensus based optimization with personal best}\label{s:modelling}
In this section we discuss how personal best information can be included in consensus based optimization
algorithms as proposed by Carrillo and co-workers in \cite{CBO, CBO2, Jose}. We start by introducing the notation before continuing with the modeling.

\subsection{Notation}
We refer the euclidean norm by $|x| = (x_1^2 + \dots + x_d^2)^{1/2}$ for $x \in \mathbb{R}^d$  and $|Y| = (\sum_{i,j=1}^{dN} Y_{ij}^2)^{1/2}$ for
matrices $Y \in \mathbb R^{dN \times dN}$. The set of natural numbers without $0$ is denoted by $\mathbb N^* ={1,2,3,\dots}$ and the half-line $[0,\infty)$ by $\mathbb R^+$. 
A vector valued function or vector $x \in \mathbb R^{dN}$ is assumed to be of the form $x = (x^1,\dots,x^N )$ with $x^i \in \mathbb R^d.$
When discussing the stochastic systems we follow the notation of \cite{Pardoux}: $(\Omega, \mathcal F, \mathbb P, \{ \mathcal F_t \}_{t\ge 0})$ corresponds to the
stochastic basis with sample space $\Omega,$ filtration $\mathcal F$ and probability function $\mathbb P.$ Moreover, $S_d^p[0,T]$ is the space of (equivalence classes of) $\mathcal P$ -measurable continuous stochastic processes $X \colon \Omega \times [0,T] \rightarrow \mathbb R^{dN}$  such that
\[
\mathbb E \sup\limits_{t \in [0,T]} |X_t|^p < +\infty \quad \text{ if } p>0.
\]
Two processes $X, Y$ are called equivalent if $(X_t = Y_t \forall t \in [0,T])$ $\mathbb
 P$-almost surely ($\mathbb
 P$-a.s.). Furthermore, $S_d^p$ is the space of (equivalence classes of) $\mathcal P$-measurable continuous stochastic processes 
 $X \colon \Omega \times \mathbb R_+ \rightarrow \mathbb R^{d}$  such that for all $T > 0$ the restriction $X_{|[0,T]}$ of $X$ to $[0,T]$ belongs to 
 $S_d^p[0,T]$. Analogously, we define $\Lambda_d^p(0,T)$ as the space of (equivalent classes) of $\mathcal P$-measurable processes $X \colon (0,T[ \rightarrow \mathbb R^d$ 
 such that
 \[
 \int_0^T |X_t|^2 dt < +\infty\quad \mathbb P\text{-a.s.} \omega \in \Omega \text{ if } p = 0 \quad\text{and}\quad  \mathbb E \left( \int_0^T |X_t|^2 dt \right)^{p/2} < +\infty\quad  \text{ if } p > 0.
 \]
 We refer to $\Lambda_d^p$ as the space of (equivalence classes of) $\mathcal P$-measurable continuous stochastic processes 
$X \colon \Omega \times (0,+\infty) \rightarrow \mathbb R^{d}$ for which for all $T > 0$ the restriction $X_{|[0,T]}$ of $X$ to $[0,T]$ belongs to $\Lambda_d^p(0,T).$ Moreover, for any 
 $\phi \in C(\mathbb R_+, \mathbb R^{dN})$ we define
 $$\| \phi \|_t := \sup\limits_{0 \le s \le t} |\phi(s)|= \sup\limits_{0 \le s \le t} \left(\phi_1(s)^2 + \dots + \phi_{dN}(s)^2\right)^{1/2}.$$
 
\subsection{The model}
We wish to approximate the global minimum 
\begin{align}
 \min_{x \in \mathbb{R}^d} f(x),
\end{align}
of a given non-negative, continuous objective function $f: \mathbb{R}^d \rightarrow \mathbb{R}$.
In doing so we consider $N \in \mathbb{N}$ particles and 
denote the position of the $i$-th particle at time $t$ by $X_t^i :=\Xi \in \mathbb{R}^d$, $i=1, \ldots N$. Note that we use $X_t = X(t) = (X^1(t), \ldots X^N(t)) \in \mathbb{R}^{dN}$, when referring to the
 positions of all particles at time $t$. In CBO  particles compare their current function value with a weighted mean value based on the current information of the whole system. A particle moves  towards the position of the weighted mean, if the function value of the weighted mean is lower.  Following the ideas of \cite{CBO,CBO2}, we use the weighted average 
\begin{align}\label{e:vf}
 v_f(t) = v_f[X_t] = \frac{\sum_{i=1}^N X^i(t) \exp(-\alpha f(X^i(t)))}{\sum_{i=1}^N \exp(-\alpha f(X^i(t)))},
\end{align}
with $\alpha > 0$, to approximate the global best, that is,~the particle with the lowest function value. Note that even though the weighted average uses only information of the current time step, it is assumed to approximate the global best over time as well, since a particle that is close to the weighed average experiences only small drift and diffusion.
The parameter $\alpha$ scales the influence of local and global minima in the weighted mean. In fact, for $\alpha = 0$ the weights are independent of the function values, 
and all particles are weighted equally. For $\alpha > 0$ the particle with the best function value has the largest weight. Moreover, the Laplace principle from large deviations theory \cite{DemboZeitouni} assures that $v_f(t)$ converges to the global best, as $\alpha \rightarrow \infty.$ 
For more details on the Laplace principle in the CBO context, we refer to \cite{CBO, CBO2}.

In the original version of PSO, see \cite{Krikpatrick}, particles compare their current position with the global best as well as their personal best value up to
that time. We propose two
different approaches how to include the personal best $p_i$ of the i-th particle. First, we consider the true personal best by setting
\begin{align}\label{e:pb}
 P_f^i(t) = \argmin_{Y \in \lbrace X^i(s) \; \colon  s \in [0,t]\rbrace} f(Y).
\end{align}
Moreover, the personal best can be approximated similarly to the global best, $v_f(t)$, defined in \eqref{e:vf}. Hereby, we use the entire trajectory in the past and refer to this trajectory by $X = (X^1,\dots, X^N)$ with $X ^i \in C(\mathbb R_+, \mathbb R^d)$ for all $i=1,\dots,N$.  Let $X^i_0$ denote
the initial position of the $i$-th particle at time $t=0$, the weighted mean over time of the $i$-th particle is defined by
\begin{align}\label{e:wpb}
p_f^i(t) = \begin{cases} X_0^i, & t = 0, \\\int_0^t X_s^i \exp(-\beta f(X_s^i)) ds \Big/ \int_0^t \exp(-\beta f(X_s^i)) ds, &\text{otherwise,} \end{cases} 
\end{align}
with  $\beta > 0$. Note that the well-posedness result presented in Section \ref{s:analysis} holds for the weighted personal best \eqref{e:wpb} only. Again, 
by the Laplace principle, we expect that $p_f^i(t) \rightarrow P_f^i$ as $\beta \rightarrow \infty$.\\
We recall that particles either move towards the global or personal best state. 
The respective CBO dynamics for the $i^{\text{th}}$ particle, $i=1, \ldots N$ are then given by the following SDE:
	\begin{equation}\label{eq:particlePSO}
	dX^i(t) = \left[ -\lambda(t,X) (X^i(t)-\vf)- \mu(t,X) (X^i(t)-\pit) \right]\,d t 
	 + \sqrt{2} \sigma \text{diag}(X^i(t) - \vf)\, dB_t^i, 
\end{equation}
where
\begin{align*} 
\lambda(t,X) &=  \Hxv\, \Hpv,\\ \mu(t,X) &= \Hxp\, \Hvp.
\end{align*}
The function $\text{H}$ corresponds to the Heaviside function and $\sigma > 0$ denotes the standard deviation.
System \eqref{eq:particlePSO} is supplemented with the initial condition $X_0^i = \xi_i$, $i=1,\dots,N.$ 
The drift and diffusion are motivated by the following considerations:
\begin{enumerate}[nosep]
	\item If the global best $\vf$ is better than the current position $\Xit$ \text{and} the personal best $\pit$, the particle moves towards the current global best $\vf$.
	\item If the personal best $\pit$ is better than the current position $\Xit$ \text{and} the global best $\vf$, the particle moves towards the personal best $\pit$.
	\item If none of the above holds, the particle still explores the function landscape via Brownian motion until it reaches the global best $\vf$.
\end{enumerate}
Note that the drift coefficients depend on the past of each particle, hence system \eqref{eq:particlePSO} is non-Markovian. The form of the memory does not allow us
to use existing results, such as \cite{Pardoux} to rewrite the system. Hence the existence and form of the respective mean-field model is, up to the authors' knowledge, not known.
\begin{remark}
(1) The CBO version proposed in \cite{Jose} can be recoverd by setting
\begin{align}\label{e:cbo}
\lambda(t,X) \equiv \lambda, \qquad \mu(t,X) \equiv 0.
\end{align}
(2) Note that the personal best \eqref{e:pb} and weighted personal best \eqref{e:wpb} can be computed very efficiently due to their accumulative structure;
this does not significantly increase 
the computational cost.
\end{remark}
Throughout this manuscript we will refer to the dynamics defined by \eqref{eq:particlePSO} with \eqref{e:cbo} as CBO, and to \eqref{eq:particlePSO} with \eqref{e:pb} or \eqref{e:wpb}
as personal best (PB) or weighted personal best (wPB), respectively.

\subsection{Toy example: CBO vs. PB dynamics}\label{s:toy}
In the following we will illustrate the differences between CBO and (w)PB using a 1D toy objective function $f$ and $3$ particles.
We consider a double well-type $f$ of the form:
\[
f(x) = (x^2-1)^2 + 0.01 x + 0.5.
\]
For this function, shown in Figure~\ref{fig:Setting1} the global and local minimum, located at  $x =-1.00125$ and $x = 0.998748$ respectively, are very close. 
In the following we perform $1000$ Monte Carlo (MC) simulations with deterministic initial conditions $\xi $.
We count a run as run successful, 
if the final position of the particles satisfy $|v_f(T) - X^i(T)| < 0.4$ for all $i=1,\dots,N.$ The final time is set to $T=100,$ the time step size 
$dt = 10^{-3}$ and $\beta$ in \eqref{e:wpb} to $\beta = 30$. 
We study the dynamics for the following two initial conditions:
\begin{enumerate}[label = (IC\arabic*), nosep, leftmargin=40pt]
	\item Initialize 2 particles near the local minimizer and 1 particle near the global minimizer. \label{i:IC1}
	\item Initialize 1 particle near the local minimizer and 2 particles near the global minimizer. \label{i:IC2}
\end{enumerate}

\begin{minipage}{0.55\textwidth}
\begin{figure}[H]
\centering
	\includegraphics[scale=0.5]{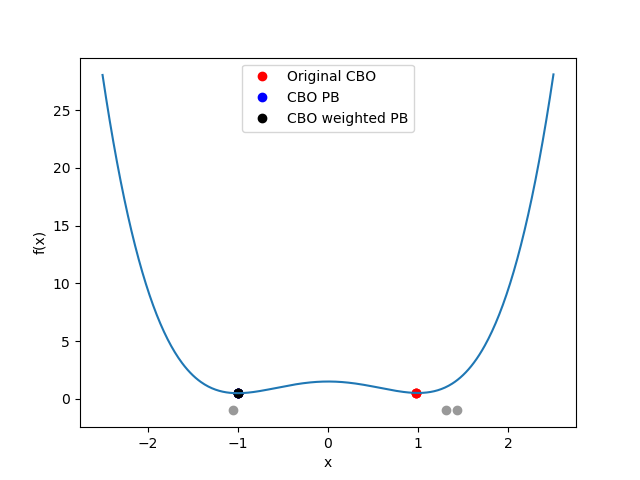}%
	\caption{ Corresponding to (IC1) Initial positions are depicted in gray. Points in different colors show $v_f(T)$ at $T= 100.$}
	\label{fig:Setting1}	
	\end{figure}
\end{minipage}
\begin{minipage}{0.4\textwidth}
\begin{table}[H]
\caption{Success rates}\label{t:toyIC1}
\begin{tabular}{ l | c | c }
	scheme & success rate & success rate \\ 
	& $\alpha =10$ &  $\alpha =30$ \\\hline
	 CBO & 30 \% &  60,9\% \\
	 PB & 100 \% & 100 \% \\
	 wPB & 100 \% & 100  \% 
\end{tabular}	
\end{table}
\end{minipage}	
\vspace{1em}

The initial positions \ref{i:IC1} and \ref{i:IC2} of the particles correspond to the gray dots in Figure~\ref{fig:Setting1} and in Figure~\ref{fig:Setting2}, respectively. We discuss \ref{i:IC1} first.  In this situation the weighted average, $v_f(0),$ is located near $x=0.9,$ thus, the Heaviside functions are zero 
and the system would be in a stationary state for $\sigma = 0$. \\
For $\sigma > 0,$ the diffusion term drives the dynamic and the particles are exploring their neighborhood. Due to the multiplicative factor, the particle on the left is exposed to more diffusion than the particles on the right. In case of the CBO scheme, the particle on the left has a high probability of
jumping out of the basin of the global minimum. Then, all particles concentrate near the local minimum. For one run, this behavior is illustrated by the positions of $v_f(T)$ shown in 
Figure~\ref{fig:Setting1} (left). This alone does not reflect the concentration which becomes apparent in Figure~\ref{fig:S1}. In fact, the orange lines show fluctuations for small times but stabilize quickly indicating that no diffusion is present and thus that all particles are concentrated. This behavior changes when personal best information is included. Here, particles still explore their neighborhood, however
at some point their current positions are worse than their personal best, and hence the drift starts pulling them back towards their personal best. This behavior is also illustrated
by the success rates stated in Table~\ref{t:toyIC1}. We see that (w)PB outperform PB for large and small values of $\alpha$. The 'pull-back' effect slows
down the convergence of (w)PB - we observe that the respective energies decrease slower than for CBO in Figure \ref{fig:S1}. Nevertheless, they find the global minimum.
\begin{figure}[htp]%
\centering
	\includegraphics[scale=0.45]{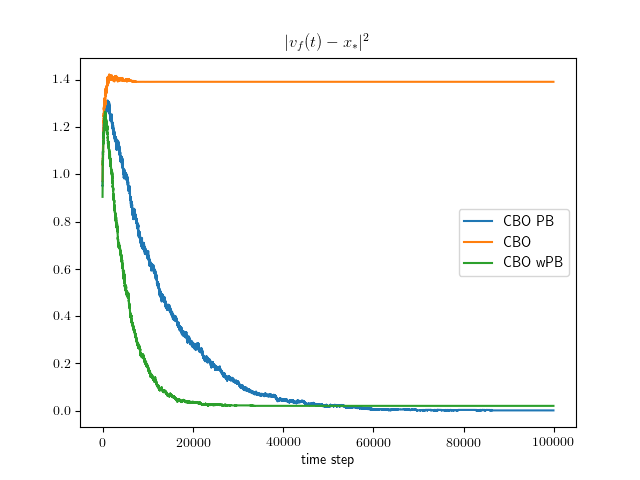}%
	\hspace*{0.5em}
	\includegraphics[scale=0.45]{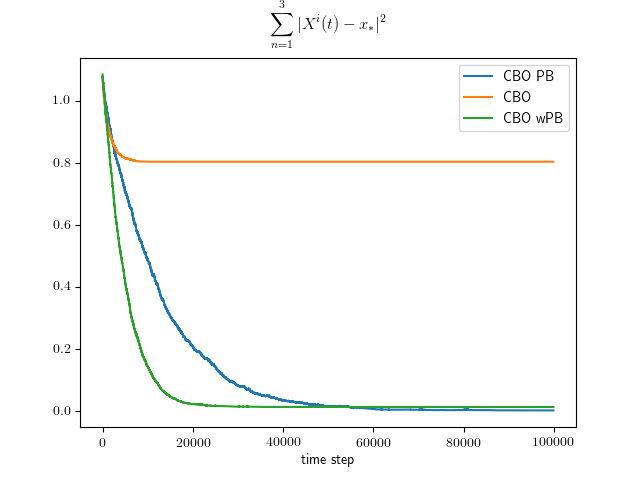}
		\caption{\ref{i:IC1} CBO is not successful while CBO with personal best finds a good approximation of the global minimizer. 
		The plot on the left shows the mean of the distances of $v_f(t)$ to the global minimizer. 
		The plot on the right shows the mean energy $\sum_{i=1}^3 |X^i(t) - x_*|^2$. The mean involves $1000$ Monte Carlo runs.}	
		\label{fig:S1}
\end{figure}

Next we consider initial condition \ref{i:IC2}. Again, in the deterministic case $\sigma = 0$ the initial configuration is stationary. For $\sigma >0$ the particles on the left are less diffusive than the particle on the right. 
Therefore, it is more likely that the particle on the right jumps into the basin of the global minimum. This is illustrated in Figure \ref{fig:Setting2} and confirmed 
by the success rates in Table \ref{t:toyIC2}. Again, CBO converges faster than (w)PB, see Figure~\ref{fig:S2}. Nevertheless, the function values at the point of concentration are smaller for (w)PB which means that the slower algorithms find better approximations. 
Note that the scale of the time step-axis is much smaller than in Figure \ref{fig:S1}.

\begin{minipage}{0.55\textwidth}
\begin{figure}[H]
\centering
\includegraphics[scale=0.4]{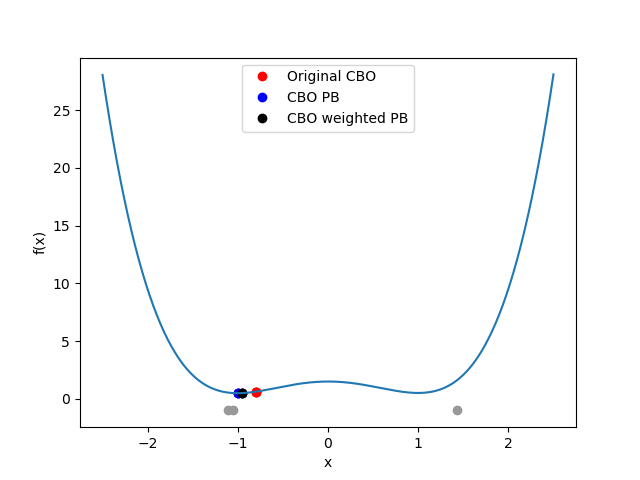}%
\caption{ Corresponding to (IC2) The initial positions are depicted in gray. The points in different colors show $v_f(t)$ at $t= 10000.$ }
\label{fig:Setting2}	
\end{figure}
\end{minipage}
\begin{minipage}{0.4\textwidth}
\begin{table}[H]	
\caption{Success rates}\label{t:toyIC2}
		\begin{tabular}{ l | c | c}
		scheme & success rate & success rate \\ 
		& $\alpha =10$ &  $\alpha =30$ \\\hline
			CBO & 91,6 \% & 98,1 \% \\
			PB & 100 \% & 100 \% \\
			wPB & 100 \% & 100 \%
		\end{tabular}	
		
		\end{table}
	\end{minipage}

\begin{figure}[htp]
\centering
	\includegraphics[scale=0.45]{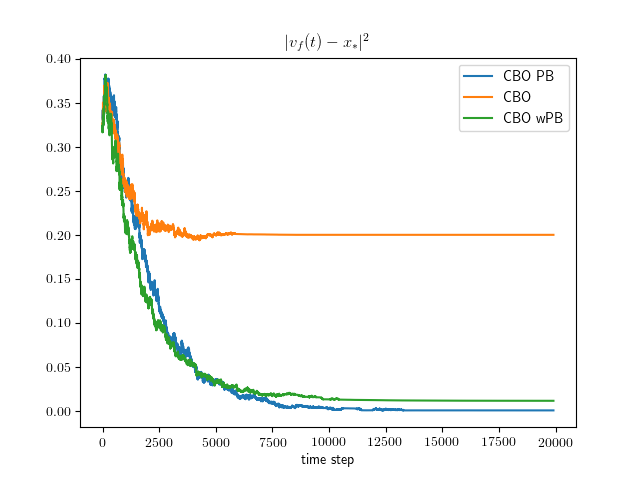}
	\includegraphics[scale=0.45]{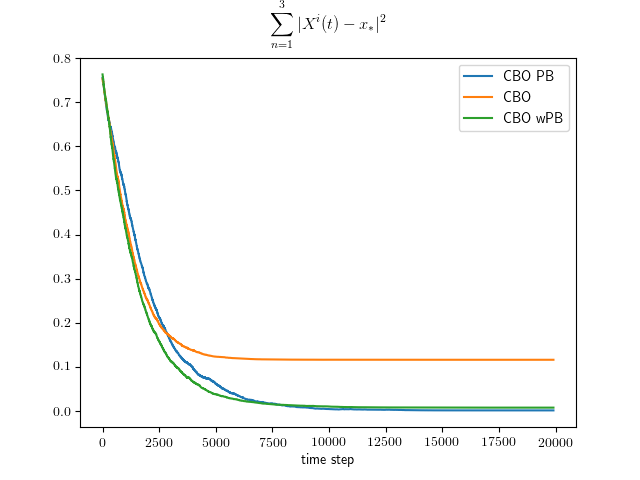}
	\caption{\ref{i:IC2} All schemes find reasonable approximations of the minimizer. The results of the methods with personal best information have a better accuracy. The plot on the left shows the mean of the distances of $v_f(t)$ to the global minimizer. The plot on the right shows the mean energy $\sum_{i=1}^3 |X^i(t) - x_*|^2$. The mean involves $1000$ Monte Carlo runs. As expected the particles following the CBO scheme are concentrating very fast. The methods with personal best information need more time for stabilization. The one with weighted personal best is slightly faster than with one with true personal best values.}
	\label{fig:S2}
\end{figure}

\section{Well-posedness results}\label{s:analysis}
In the following we discuss well-posedness of the wPB model. We begin by considering CBO with component-wise diffusion, which was proposed in \cite{Jose}.

\subsection{Well-posedness of CBO with component-wise diffusion }
\begin{theorem}\label{t:wellposed}
Let $f$ be locally Lipschitz and $N \in \mathbb{N}$. Then system \eqref{eq:particlePSO} with $\lambda(t,X) \equiv \lambda, \; \mu(t,X) \equiv 0$ admits a unique strong solution for any initial condition 
$\xi = (\xi_1, \dots, \xi_N)$ satisfying $\mathbb E\,|\xi|^2 < \infty$.
\end{theorem}	
A detailed proof can be found in the Appendix. Let us just emphasize that the estimates in the proof of Theorem \ref{t:wellposed} are independent of the dimension, $d,$ as was already highlighted in \cite{Jose} for the mean-field setting. This is in contrast to \cite{CBO,CBO2}, where the estimates depend on $d.$

\subsection{Well-posedness in case of weighted personal best}
Next, we present an existence and uniqueness result for the proposed SDE model with weighted personal best and smoothed Heaviside functions. 
Note that the structure of the weighted personal best suggests the idea of introducing
	new variables for the numerator and the denominator. This re-formulation converts the non-Markovian process into a Markovian system for times $t>0$, 
	but violates the initial condition.
	We therefore use different
 proofs for properties of SDEs with local Lipschitz conditions as well as path-dependent SDEs that can be found in the literature \cite{Pardoux}. 
To the authors' knowledge none of them covers the case of path-dependent SDEs with local Lipschitz conditions. 
In the following we present a proof which combines the two techniques to obtain a well-posedness result.

We assume that the regularized Heaviside function $H^\epsilon$ satisfies the following conditions:
\begin{enumerate}[label=(A\arabic*), nosep]
	\item Let $0 \le H^\epsilon(x) \le 1$ for all $x \in \mathbb{R}$.\label{a:Hbound}
	\item There exists a constant $C>0$ such that 
\begin{align}
	|H^\epsilon(x) - H^\epsilon(y)| \le \frac{C}{\epsilon}|x-y| \qquad \text{	for all $x,y \in \mathbb{R}.$}
\end{align}\label{a:HLip} 
\end{enumerate}
This corresponds to the following regularized problem 
\begin{subequations}\label{eq:regularisedPSO}
\begin{equation}\label{eq:regularized}
dX(t) = \left[ -\lambda^\epsilon(t,X) (X(t)-\vf)- \mu^\epsilon(t,X) (X(t)-\pit) \right]\,d t 
+ \sqrt{2} \sigma \text{diag}(X(t) - \vf)\, dB_t, 
\end{equation}
with
\begin{align} \label{eq:smoothedcoeff}
\lambda^\epsilon(t,X) &=   \text{diag}\Big(\Hexv\, \Hepv \Big)_{i=1,\dots,N} \in \mathbb{R}^{dN\times dN},\\ \mu^\epsilon(t,X) &=  \text{diag}\Big( \Hexp\, \Hevp\Big)_{i=1,\dots,N}\in \mathbb{R}^{dN\times dN}.
\end{align}
\end{subequations}
Moreover, we assume that the objective function $f$ satisfies the following properties:
\begin{enumerate}[resume,label=(A\arabic*), nosep] 
\item  Positivity: it holds $0 \le  f(x)  \quad \text{ for all } x \in \mathbb{R}^d,$
\label{a:fbound}
\item Quasi-local Lipschitz condition: for any $n < \infty$ and $|x|,|y| \le n$ it holds $$ |f(x) - f(y)| \le L_f|x-y|, $$
with a constant $L_f>0$ depending on $n$ only. \label{a:fLip}
\end{enumerate}
\begin{remark}
The well-known regularization of the Heaviside function
\[
H^\epsilon(x) = \frac{1}{2} + \frac{1}{2} \tanh\left(\frac{x}{\epsilon}\right)
\] 
satisfies the assumptions \ref{a:Hbound} and \ref{a:HLip}.
Note that in the context of optimization problems, the positivity assumption on $f$ is not too restrictive. 
Since $f$ corresponds to a minimization functional it is naturally bounded from below and can be shifted to satisfy the positivity constraint. 
\end{remark}

The following proof is based on a combination of arguments of Theorem 3.17 and Theorem 3.27 in \cite{Pardoux} - this yields well-posedness of \eqref{eq:regularisedPSO}. We  begin with two lemmata providing necessary estimates. The first lemma is concerned with properties of the weighed averages. 
Note that the global best $v_f$ depends on the current state of the process and the personal best $p_f$ on the whole trajectory. 
We therefore write $v_f[\varphi(t)] = v_f(t) $ and $p_f[\varphi] = p_f$.

\begin{lemma}\label{lem:estimatesOnPV}
	Let $f$ satisfy \ref{a:fbound} and \ref{a:fLip}, $N \in \mathbb N$  and $\varphi = (\varphi^1, \dots \varphi^N) \in \mathcal C(\mathbb R_+, \mathbb{R}^{dN}).$ Then
	\begin{align}
	\begin{aligned}
	&v_f[\varphi(t)] \in \mathbb R^d, \qquad 	&|v_f[\varphi(t)]| &\le  |\varphi(t)|  \quad\text{for every } t,\\
	&p_f[\varphi] \in \mathcal C(\mathbb R_+, \mathbb R^{dN}),  \qquad &|p_f[\varphi](t)| &\le |\varphi|_t. 
	\end{aligned}
	\end{align}
	Moreover, the averages satisfy the local Lipschitz conditions:
	\begin{gather}		
	|p_f[\varphi](t) - p_f[\hat \varphi](t)|^2 = \sum_{i=1}^N \left|p_f^i[\varphi](t) - p_f^i[\hat \varphi](t) \right|^2 \le C_1 \|\varphi - \hat \varphi\|_t^2, \\
	\left|v_f[\varphi(t)] - v_f[\hat \varphi(t)] \right|^2 \le C_2 | \varphi(t) - \hat \varphi(t)|^2
	\end{gather}
	for all $t \in [0,\infty)$ with $|\varphi|_t, |\hat \varphi|_t \le n$ with constants
	\begin{align}\label{e:Ci} 
	C_1 &= \Big( 1+ (1 + 2L_f)   \beta n e^{ \beta (\overline{f}- \underline{f})}  \Big), \text{ and } C_2 = \left( 1+ \frac{\alpha n L_f e^{-\alpha \underline{f}}}{N} + n e^{\alpha(\overline{f} - \underline{f})} \left( \frac{1}{N} + \alpha n L_f  \right) \right)^2 2^{N-1}.
	\end{align}
	Here $L_f$ is the Lipschitz constant of $f$ in $B_n = \{ x \colon |x| \le n\}$ and $\underline f, \overline f$ correspond, respectively, to the minimal and maximal values of $f$ on $B_n.$
\end{lemma}	
The proof of Lemma~\ref{lem:estimatesOnPV} can be found in the Appendix. Using Lemma \ref{lem:estimatesOnPV} we show that the drift and diffusion terms satisfy local Lipschitz and linear growth conditions. These properties allow us to apply the existence and uniqueness result later on.
\begin{lemma}
	Let \ref{a:Hbound}-\ref{a:fLip} hold. Then  
	$$b \colon [0, +\infty) \times  \mathcal C(\mathbb R_+, \mathbb R^{dN} ) \rightarrow \mathbb{R}^{dN} \text{ and } \Sigma\colon [0, +\infty) \times  \mathcal C(\mathbb R_+, \mathbb R^{dN} ) \rightarrow \mathbb{R}^{dN \times dN} $$
	given by
	$$ b(t, \varphi) = -\lambda^\epsilon(t,\varphi)(\varphi(t)-v_f) - \mu^\epsilon(t,\varphi)(\varphi(t) - p_f) $$
	and 
	$$ \Sigma(t, \varphi) = \text{diag}\left( (\varphi^{i}(t) - v_f)_{i=1,\dots,N} \right) \in \mathbb{R}^{dN \times dN}$$
	with $\varphi^i = (\varphi_{(i-1)d+1}, \dots, \varphi_{(i-1)d+d})$
	satisfy the following conditions for all $\varphi, \psi \in C(\mathbb R_+, \mathbb R^{dN})$ and all $R > 0:$ 
	\begin{multicols}{2}
	\begin{enumerate}[nosep, label=(\roman*)]
		\item \label{i:Lipb} $|b(t,\varphi) - b(t, \psi)| \le L_R  \|\varphi - \psi\|_t,$
		\item \label{i:growthb} $|b(t,\varphi)| \le  a \|\varphi\|_t $  ,
		\item \label{i:Lipsig} $|\Sigma(t,\varphi(t)) - \Sigma(t, \psi(t))| \le \ell_R  |\varphi(t)- \psi(t)|,$
		\item \label{i:growthsig} $|\Sigma(t,\varphi(t))| \le  b_R |\varphi(t)|$ ,
	\end{enumerate}
	\end{multicols}
	where $L_R, \ell_R, a_R, b_R \in \mathbb R.$ 
\end{lemma}	
\begin{proof}
	To show \ref{i:Lipb} we calculate
	\begin{equation}
	|b^i(t,\varphi) - b^i(t,\psi)|^2  \le 2(I_1 + I_2 + I_3 + I_4),
	\end{equation}
	where
	\begin{align*}
	I_1 &:= \left| (\lambda^{\epsilon, i}(t,\varphi) - \lambda^{\epsilon, i}((t,\psi) ) (\varphi^i(t) - v_f[\varphi(t)]) \right|^2 \\  &\le \frac{1}{2\epsilon} L_f( |\varphi^i(t)| + |v_f[\varphi(t)]|) \left( 2|\varphi^i(t) - \psi^i(t)|^2 + 8|v_f[\psi(t)] - v_f[\varphi(t)] |^2 + 4| p_f^i[\varphi](t) - p_f^i[\psi](t) |^2 \right), \\
	I_2 &:= \left| \lambda^{\epsilon, i}((t,\psi) (|\varphi^i(t) - \psi^i(t)| + |v_f[\psi(t)] - v_f[\varphi(t)]|) \right|^2 \le 2|\varphi^i(t) - \psi^i(t)|^2 + 2|v_f[\psi(t)] - v_f[\varphi(t)|^2, \\
	I_3 &:= \left| (\mu^{\epsilon, i}((t,\varphi) - \mu^{\epsilon, i}((t,\psi) ) (\varphi^i(t) - p^i_f[\varphi](t)) \right|^2 \\ &\le \frac{1}{2\epsilon}L_f (|\varphi(t) - p_f^i[\varphi](t)|) \left( 2|\varphi^i(t) - \psi^i(t)|^2 + 4|v_f[\psi(t)] - v_f[\varphi(t)] |^2 + 8| p_f^i[\varphi](t) - p_f^i[\psi](t) |^2 \right) , \\
	I_4 &:= \left| \mu^{\epsilon, i}((t,\psi) (\varphi^i(t) - \psi^i(t) + p^i_f[\psi](t) - p^i_f[\varphi](t)) \right|^2 \le 2|\varphi^i(t) - \psi^i(t)|^2 + 2|p_f^i[\varphi](t) - p_f^i[\psi](t)|^2.
	\end{align*}
From Lemma \ref{lem:estimatesOnPV} we know that $|v_f[\psi(t)] - v_f[\varphi(t)] |^2 \le C_1 |\varphi(t) - \psi(t)|^2,$ and 
 $|p_f^i[\varphi](t) - p_f^i[\psi](t)|^2 \le C_2 \|\varphi^i - \psi^i\|_t^2$
 	with constants $C_1$ and $C_2$ given by \eqref{e:Ci} with $n = R$.
 	This yields \ref{i:Lipb} since
 	\begin{align*}
 	|b(t,\varphi) - b(t, \psi)| = \left( \sum_{i=1}^N |b^i(t,\varphi) - b^i(t,\psi) | \right)^{1/2} \le L_R \|\varphi - \psi\|_t.
 	\end{align*}
 	The Lipschitz bound \ref{i:Lipsig} follows from similar arguments using the diagonal structure of $\Sigma$:
 	\begin{align*}
 	|\Sigma(t,\varphi(t))- \Sigma(t,\psi(t))| &= \left( \sum_{i=1}^N |\Sigma_{ii}(t,\varphi(t))- \Sigma_{ii}(t,\psi(t))|^2\right)^{1/2} \\ &\le \left( \sum_{i=1}^N 2|\varphi^i(t)- \psi^i(t)|^2 + 2|v_f[\varphi(t)]- v_f[\psi(t)]|^2 \right)^{1/2} \le \ell_R |\varphi(t)- \psi(t)|.
 	\end{align*}
 	The last two inequalities hold due to
 	\begin{align*}
 		|b(t,\varphi)| &= \left( \sum_{i=1}^N b^i(t,\varphi)^2 \right)^{1/2} \le \left(\sum_{i=1}^N 8|\varphi^i(t)|^2 + 4|v_f[\varphi(t)]|^2 + 2|p_f^i[\varphi](t)|^2 \right)^{1/2} \le a_R \| \varphi \|_t, \\
 		|\Sigma(t,\varphi)|	&= \left( \sum_{i=1}^N \Sigma_{ii}(t,\varphi)^2 \right)^{1/2} \le \left( \sum_{i=1}^N 2 \varphi^i(t)^2 + 2 |v_f[\varphi(t)]|^2 \right)^{1/2} \le b_R |\varphi(t)|.
 	\end{align*}
\end{proof}
Equipped with this lemma, we have everything at hand to prove the main theorem.
\begin{theorem} Let \ref{a:Hbound}-\ref{a:fLip} be satisfied and $\xi \in L^0(\Omega, \mathcal F_0, \mathbb P, \mathbb R^{dN})$ with $\mathbb E|\xi|^p < \infty$ for each $p>0.$ Then, there exists a unique strong global solution to \eqref{eq:regularisedPSO}. Moreover, there exists a constant $C_{p,T,L_r, \ell_R}$ such that
\[
\mathbb E \sup\limits_{t \in [0,T]} |X(t)|^p \le C_{p,T,L_r, \ell_R}\mathbb E |\xi|^p. 
\]
\end{theorem}
The proof combines arguments of Theorem 3.17 (path-dependent SDE) and Theorem 3.27 (SDE with local Lipschitz coefficients) in \cite{Pardoux}. 
\begin{proof}
	We start by proving uniqueness. Let $X, \hat X \in S_{dN}^0$ be two solutions to \eqref{eq:regularisedPSO} corresponding to initial data 
	$\xi, \hat \xi \in L^0(\Omega, \mathcal F_0, \mathbb P, \mathbb R^{dN}),$ respectively. Then it holds $\mathbb E\,\xi = \mathbb E\,\hat \xi$. Define the stopping time
	$\tau_n(\omega) = \inf \{ t \ge 0 \colon |X_t(\omega)| + |\hat X_t(\omega)| \ge n\}.$ Then the two solutions satisfy
	\begin{align*}
		X_{t \wedge \tau_n} &= \xi + \int_0^t 1_{[0,\tau_n]}(s) b(s \wedge \tau_n, X) ds + \int_0^t 1_{[0,\tau_n]}(s) \sigma(s \wedge \tau_n, X_{s \wedge \tau_n}) dB_s, \\
		\hat X_{t \wedge \tau_n} &= \hat \xi + \int_0^t 1_{[0,\tau_n]}(s) b(s \wedge \tau_n, \hat X) ds + \int_0^t 1_{[0,\tau_n]}(s) \sigma(s \wedge \tau_n, \hat X_{s \wedge \tau_n}) dB_s,
	\end{align*} 
	and $\tau_n \rightarrow \infty$ for $n \rightarrow \infty.$
	Note that we have global Lipschitz constants for $b$ and $\sigma$ for all times $t \in [0, \tau_n]$ with $n$ arbitrary but fixed. This allows us to use Theorem 3.8 in \cite{Pardoux}, see proof of Theorem 3.27 in \cite{Pardoux} for more details, to obtain 
	\[
	\mathbb E \frac{\| e^{-V^R} (X_{\cdot \wedge \tau_n} - \hat X_{\cdot \wedge \tau_n}) \|_{T}^p}{\left( 1 + \| e^{-V^R} (X_{\cdot \wedge \tau_n} - \hat X_{\cdot \wedge \tau_n}) \|_{T}^2 \right)^{p/2}} \le C_{p} \mathbb E \frac{|X_0 - \hat X_0|^p}{\left(1 + |X_0 - \hat X_0|^2 \right)^{p/2}},
	\] for some $V^R$ depending on the local Lipschitz constants $L_R, \ell_R, a_R, b_R$ and $p \ge 2$ arbitrary. Hence, the uniqueness of the solution on $[0,\tau_n].$ As $\tau_n \rightarrow \infty$ for $n \rightarrow \infty,$ this allows us to conclude the global uniqueness of $X \in S_{dN}^0.$ 
	
	Next, we show the existence of solutions. Let $M \in \mathbb N^*$ and $0 = T_0 < T_1 < \dots < T_M = \tau_n$ with $T_i = \frac{i \tau_n}{M}$.  It holds
	\[
	\alpha(\frac{\tau_n}{M}) := \sup\limits_{0 < s-t < \frac{\tau_n}{M}} \left( \int_t ^s L_R \; dr \right)^p + \left( \int_t^s \ell_R^2\, dr \right)^{p/2} \longrightarrow 0 \quad \text{ as } M \rightarrow \infty.
	\]
	We employ a fixed point argument for the mapping $\Gamma \colon S_{dN}^p [0,T_1] \rightarrow S_{dN}^p[0, T_1]$ given by
	\[
	\Gamma(U)_t = \xi + \int_0^t b(s, U) ds + \int_0^t \sigma(s, U_s) dB_s,
	\]
	where we need no stopping times due to $t < \tau_n$ on $[0,T_1].$ Indeed, the mapping $\Gamma$ is well-defined since for all $\varphi \in C(\mathbb R_+, \mathbb R^{dN})$ 
	\[
	|b(t,\varphi) | \le L_R \|\varphi\|_t, \textrm{ and } |\sigma(t,\varphi) | \le \ell_R \|\varphi\|_t,
	\]
	is satisfied. Because of the Lipschitz continuity, both stochastic processes $b(\cdot,U)$ and $\sigma(\cdot, U_s)$ are progressively measurable for all $U \in S_{dN}^p[0,\tau_n]$ and $b(\cdot, U) \in L^p(\Omega, L^1(0,\tau_n))$ and $\sigma(\cdot, U) \in \Lambda^p_{dN \times dN}(0, \tau_n).$ Therefore,
	\[
	\int_0^\bullet b(r, U) dr\, ,\int_0^\bullet \sigma(r, U_r) dB_r  \; \in S_{dN}^p[0,\tau_n]. 
	\]
	We will show that the operator $\Gamma$ is a strict contraction on the complete metric space $S_{dN}^p[0,T_1]$ for sufficiently large $M$ (where $S_{dN}^p[0,T_1]$ is
	equipped with the usual distance
	$d_{p,M}(U,V) = (\mathbb E\| U -V\|_{T_1}^p)^{1/{p \vee 1}}$).
	Let $U,V \in S_{dN}^p[0,T_1].$ By the Burkholder-Davis-Gundy inequality we have
	\begin{align*}
	\mathbb E \| \Gamma(U) - \Gamma(V) \|_{T_1}^p &\le (1 \vee 2^{p-1})\, \mathbb E \sup\limits_{s \in [T_0,T_1]} |\int_{T_0}^s b(r, U) - b(r,V) dr|^p \\
	 & \qquad + (1 \vee 2^{p-1})\, \mathbb E \sup\limits_{s \in [T_0,T_1]} |\int_{T_0}^s \sigma(r, U_r) \sigma(r,V_r) dB_r|^p \\
	 & \le  (1 \vee 2^{p-1})\, \left[ \mathbb E \left( \int_{T_0}^{T_1} L_R \| U -V \|_r dr \right)^p + \mathbb E \left( \int_{T_0}^{T_1} \ell_R^2 | U_r - V_r | dr  \right)^{p/2} \right] \\
	 & \le  (1 \vee 2^{p-1})\, \alpha (\frac{\tau_n}{M})\, \mathbb E(\| U -V \|_{T_1}^p).
	\end{align*}
	Let $M_0 \in \mathbb N^*$ such that $ (1 \vee 2^{p-1}) \alpha(\frac{\tau_n}{M_0}) \le \left( \frac{1}{2} \right)^{1 \vee p}$. Then $\Gamma$ is a strict contraction in $S_{dN}^p[0,T_1]$ and thus \eqref{eq:regularisedPSO} has a unique solution $X \in S_{dN}^p[0,T_1].$
	We extend the solution to the interval $[0,T_2]$ by defining a mapping, again, called $\Gamma \colon S_{dN}^p[0,T_2] \rightarrow S_{dN}^p[0,T_2]$:
	\[
	\Gamma(U)_t = \begin{cases} X_t, &\text{if } t \in [0,T_1], \\
	X_{T_1} + \int_{T_1}^t b(s, U) ds + \int_{T_1}^t \sigma(s,U_s) ds, & \text{if } t \in (T_1, T_2]. \end{cases}
	\]
	We repeat the argument $M_0$ times to be valid the whole interval $[0, \tau_n]$. Since $\tau_n \rightarrow \infty$ almost surely, the uniqueness of the solution implies that 
	\[
	[X_t^{n+1}(\omega) - X_t^n(\omega)]1_{[0,\tau_n(\omega)]}(t) 1_{[0,\infty)}(\tau_n(\omega)) = 0.
	\]
	Hence, the process $X \in S_d^0$ is defined by $X_t(\omega) = X_t^n(\omega)$ if $0 \le t \le \tau_n(\omega)$ and $\tau_n(\omega) > 0$ is the unique solution to the regularized problem \eqref{eq:regularisedPSO}.
\end{proof}

\section{Numerical results}\label{s:numerics}
The numerical simulations are based on the direct simulation of system \eqref{eq:particlePSO} using the Euler-Maruyama scheme, in which we do not
approximate the Heaviside function $H$. Note that the smoothing of $H$ was only needed for analytic considerations. In practise, we want only
one of the drift terms to effect the particles dynamics, which is why we have not pursued this option any further. The final time is set to $T = 15$, discretized into $3 \times 10^4$ time steps. All 
presented results are averaged over $M = 5000$ realizations. The standard deviation is set to
$\sigma = 0.5$, while the number of agents depends on the dimension of the function space. In particular, we set the number of agents to $3,5$ or $10$ times the space
dimension. The initial positions of particles are drawn from a uniform distribution within a specific domain for each function. The parameters $\alpha$ and $\beta$ to compute the
global and personal best are set to
\begin{align*}
\alpha = 10 \text{ and } \beta = 10.
\end{align*}
A realization is successful if the average mean is close to the function minimum $f_{\min}$, in particular
\begin{align*}
 \lvert f(v_f(T)) - f(x_{\min})\rvert < 0.1.
\end{align*}
We compare the performance of the CBO scheme with $\mu=0$, PB and wPB for the following benchmark problems:
\begin{enumerate}[nosep]
\item \textit{Alpine \cite{Benchmark}:} This non-convex differentiable function 
has a global minimum at $x_{\min} = (0,\ldots, 0)$
\begin{align}\label{e:alpine}
f(x) = \sum_{i=1}^d \lvert x_i \sin(x_i) + 0.1 x_i\rvert.
\end{align}
\item \textit{Ackley \cite{Ackely}: } This function is continuous, non-differentiable and non-convex and has its global minimum at $x_{\min} = (0, \ldots ,0)$. 
\begin{align}\label{e:ackley}
f(x) = -20\exp(0.2\sqrt{\frac{1}{d}|x|^2}) - \exp(\frac{1}{d}\sum_{i=1}^d \cos(2\pi x_i)) + 20 + \exp(1).
\end{align}
\item \textit{Rastrigin \cite{Rastrigin}:} The Rastrigin function is continuous, differentiable and convex, has lots of local minima and a global minimum at $x_{\min} = (0,\ldots 0)$.
\begin{align}\label{e:rastrigin}
 f(x) = 10d + \sum_{i=1}^d (x_i^2 - 10 \cos(2 \pi x_i)).
\end{align}
\item \textit{Xinsheyang2: \cite{Benchmark} } This function is continuous but not differentiable and non-convex with a global minimum at $x_{\min} = (0, \ldots 0)$.
\begin{align}\label{e:xinsheyang}
f(x) = \sum_{i=1}^d \lvert x_i \rvert \exp(-\sum_{i=1}^d \sin(x_i^2)).
\end{align}
\end{enumerate}
The choice of these functions is based on the different characteristics they have, see Figure~\ref{fig:benchs} for plots in 2D.
\begin{figure}[htp]
	\centering
	\includegraphics[scale=0.2]{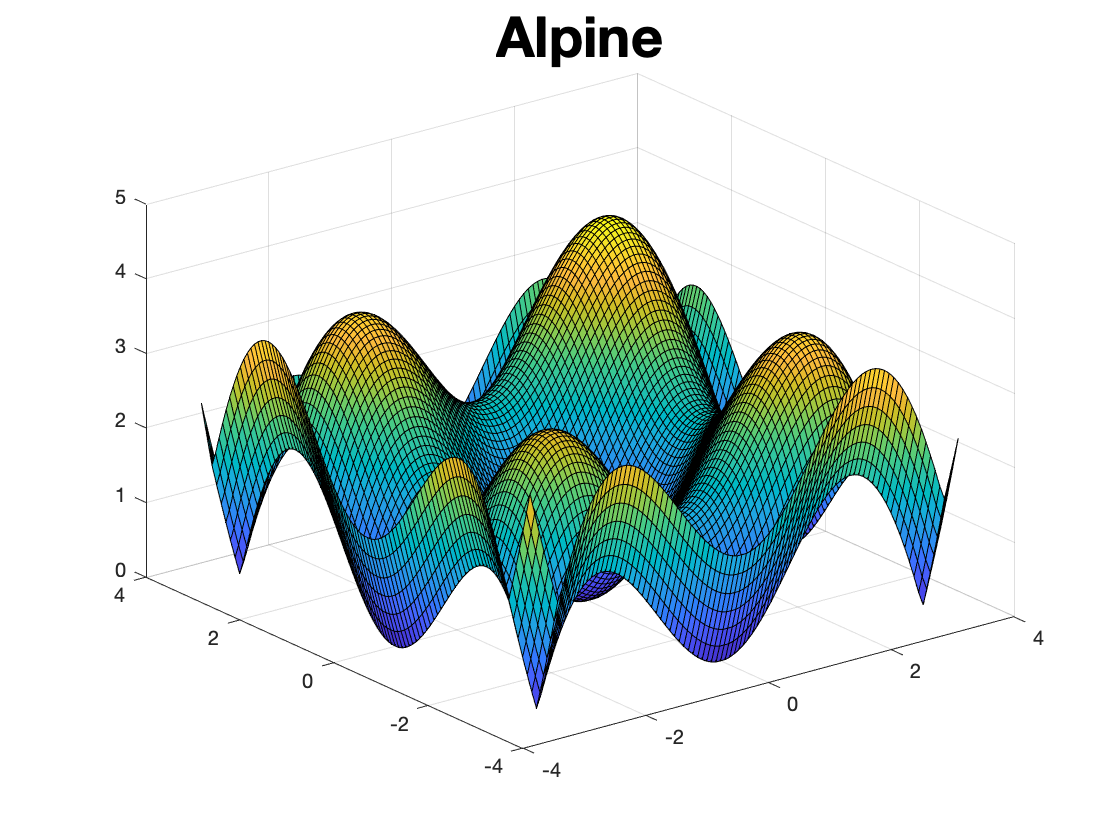}
	\includegraphics[scale=0.2]{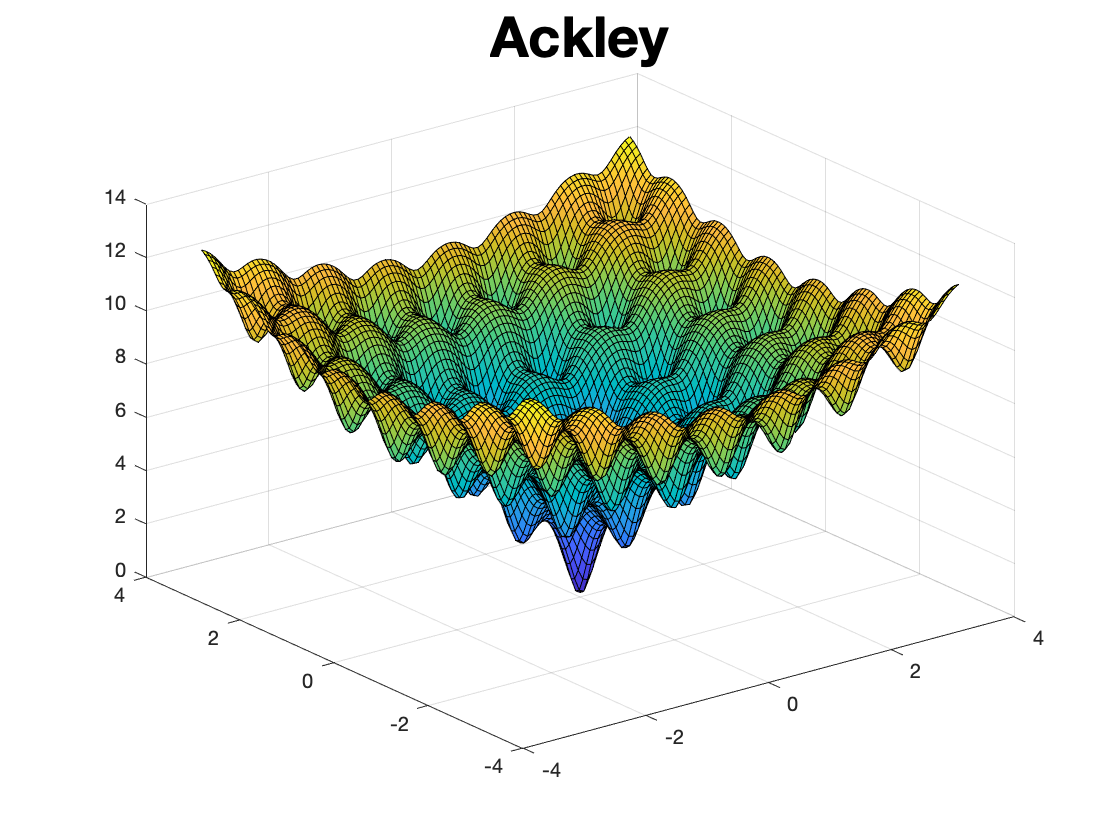} \\
	\includegraphics[scale=0.2]{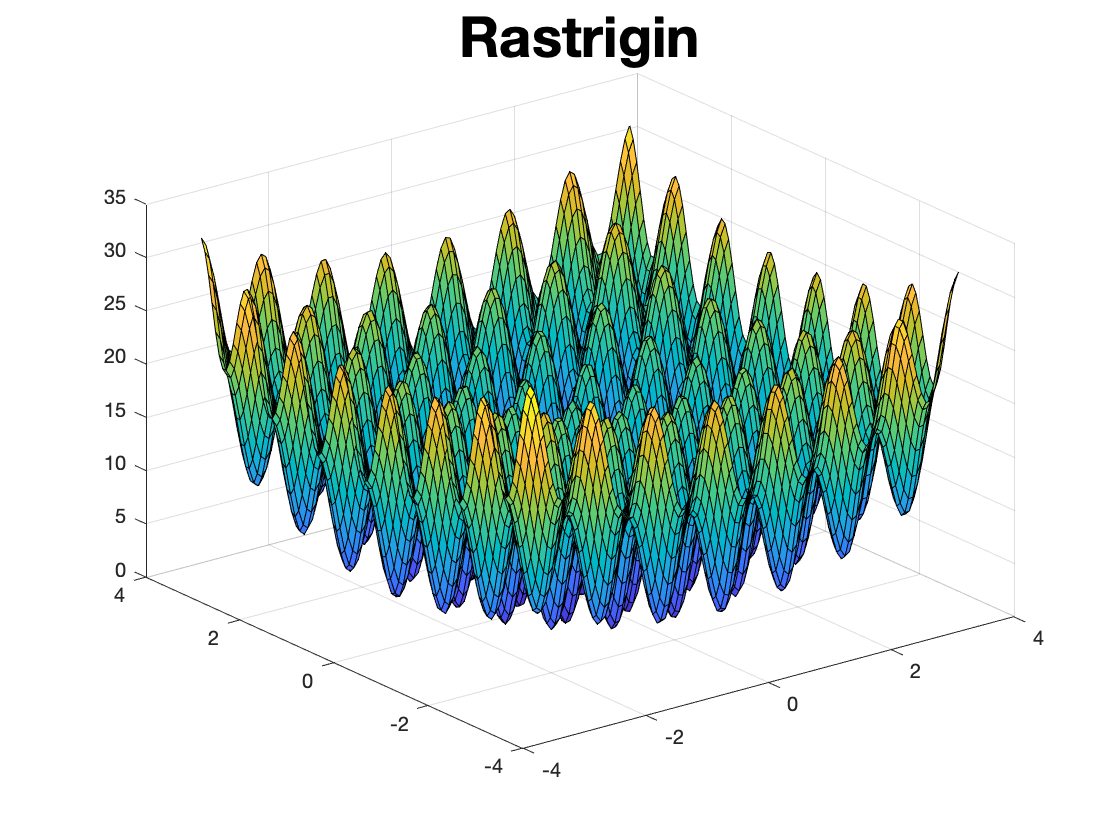}
	\includegraphics[scale=0.2]{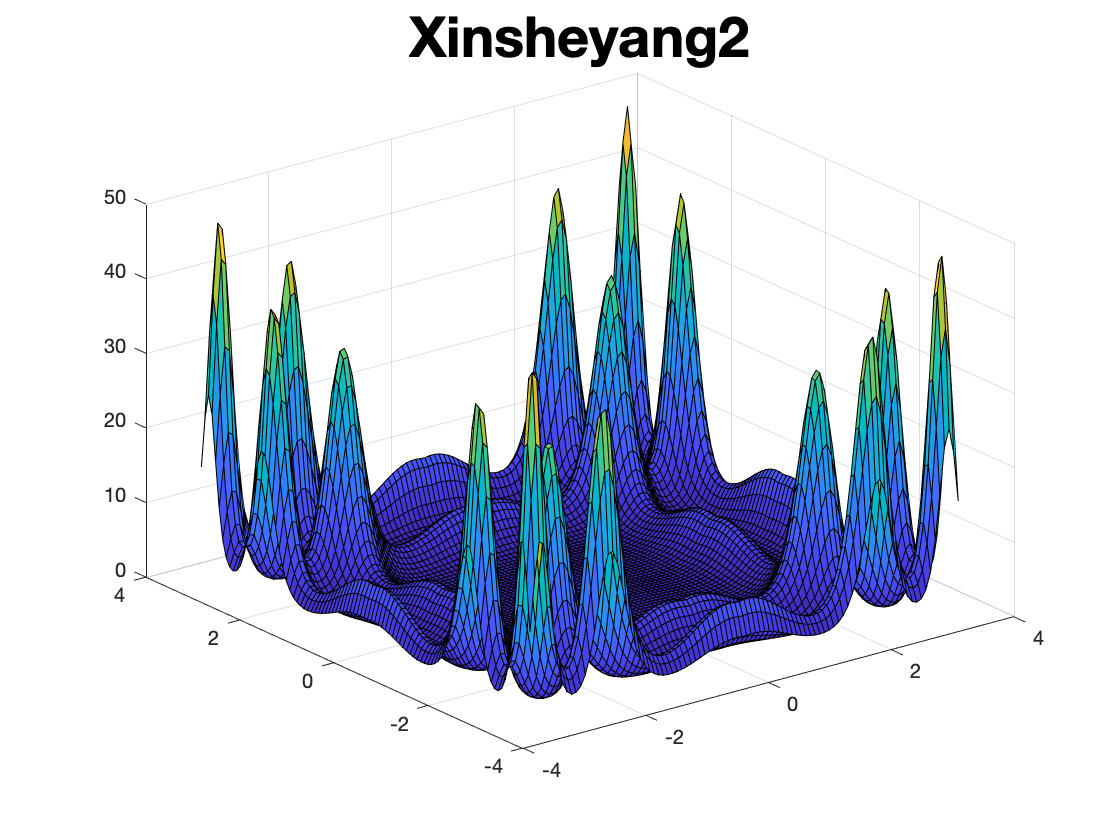}
	\caption{2D plots of the benchmark functions.}
	\label{fig:benchs}
\end{figure}	
Table \ref{t:alpineackley} shows the results for the Alpine function \eqref{e:alpine} and the Ackley function \eqref{e:ackley}. 
We observe that the success rate increases with the number of particles,
and decreases for higher space dimension. Weighted personal best and personal best give comparable results, which is not surprising 
since wPB approximates PB for large values of $\beta$. 
A similar behavior can be seen in the case of the  Rastrigin function \eqref{e:rastrigin} and the Xinsheyang function \eqref{e:xinsheyang} in Table \ref{t:rastxinshe}.
\csvstyle{myTableStyle}{tabular=|c|c|c|c|c|,
table head=\hline\textbf{d} & \textbf{\# par.} & \textbf{CBO} & \textbf{PB} & \textbf{wPB}\\\hline\hline,
late after line=\\\hline,head to column names}
\begin{table}[htp]
\centering
\begin{minipage}{0.45\textwidth}
\begin{filecontents*}{alpine.csv}
Dim,Nparticles,CBO,PB,WPB
1,3,0.8372,0.7184,0.8456
1,5,0.9626,0.8424,0.968
1,10,0.9986,0.9292,0.999
3,9,0.4868,0.4904,0.485
3,15,0.6294,0.6754,0.6458
3,30,0.847,0.8858,0.8404
5,15,0.3246,0.3266,0.3194
5,25,0.4206,0.4352,0.4226
5,50,0.5748,0.6092,0.5778
\end{filecontents*}
\csvreader[myTableStyle]{alpine.csv}{}%
{\Dim&\Nparticles~& \CBO &\PB & \WPB}
\end{minipage}
\begin{minipage}{0.45\textwidth}
\begin{filecontents*}{ackley.csv}
Dim,Nparticles,CBO,PB,WPB
1,3,0.8138,0.8164,0.8144
1,5,0.9598,0.9602,0.9598
1,10,0.9986,0.999,0.9986
3,9,0.902,0.917,0.9008
3,15,0.9908,0.9934,0.99
3,30,1,1,1
5,15,0.9886,0.9928,0.9908
5,25,0.9996,1,0.9998
5,50,1,1,1
\end{filecontents*}
\csvreader[myTableStyle]{ackley.csv}{}%
{\Dim&\Nparticles~& \CBO &\PB & \WPB}
\end{minipage}
\caption{Success rates of consensus based optimization (CBO), personal best (PB) and weighted personal best (wPB) scheme for Alpine \eqref{e:alpine} and Ackley \eqref{e:ackley} in space dimension $d$ for different $\#$ of particles.}
\label{t:alpineackley}
\end{table}

\begin{table}[htp]
\centering
\begin{minipage}{0.45\textwidth}
\begin{filecontents*}{rastrigin.csv}
Dim,Nparticles,CBO,PB,WPB
1,3,0.7924,0.7992,0.8006
1,5,0.9528,0.953,0.9522
1,10,0.9984,0.999,0.9988
3,9,0.4634,0.4726,0.4624
3,15,0.6844,0.7112,0.6914
3,30,0.9138,0.9278,0.9204
5,15,0.5012,0.4944,0.5
5,25,0.7074,0.7296,0.7144
5,50,0.908,0.9172,0.9106
\end{filecontents*}
\csvreader[myTableStyle]{rastrigin.csv}{}%
{\Dim&\Nparticles~& \CBO &\PB & \WPB}
\end{minipage}
\begin{minipage}{0.45\textwidth}
\begin{filecontents*}{xinsheyang2.csv}
Dim,Nparticles,CBO,PB,WPB
1,3,0.8732,0.8914,0.8762
1,5,0.9786,0.9892,0.9802
1,10,0.9996,1,0.9998
3,9,0.8538,0.8634,0.8548
3,15,0.9268,0.9412,0.9298
3,30,0.9766,0.984,0.9792
5,15,0.481,0.45,0.4742
5,25,0.4996,0.4724,0.498
5,50,0.5318,0.4918,0.5216
\end{filecontents*}
\csvreader[myTableStyle]{xinsheyang2.csv}{}%
{\Dim&\Nparticles~& \CBO &\PB & \WPB}
\end{minipage}
\caption{Success rates of consensus based optimization (CBO), personal best (PB) and weighted personal best (wPB) scheme for Rastrigin \eqref{e:rastrigin} and Xinsheyang2 \eqref{e:xinsheyang} in space dimension $d$ for different $\#$ of particles.}
\label{t:rastxinshe}
\end{table}
We conclude by investigating a 2D version of the toy problem considered in Section \ref{s:toy}:
\begin{align*}
 f(x_1, x_2) = (x_1^2 - 1)^2  + 0.01 x_1 + 0.5 + x_2^2.
\end{align*}
This function has a global minimum at $x_{\min} = (-1.00125, 0)$ and a local minimum at $(0.998748,0)$.
We wish to explore the dynamics of this 2D version for different number of particles. In doing so we consider $4$, $8$ or $16$ particles and start the particle schemes with
$2$, $4$ and $8$ placed in each of the two wells with a random perturbation. Furthermore we set $\alpha=10$ and $\beta = 20$. Table \ref{t:toy2d} shows the results as the number of particles increases. We observe that CBO and (w)PB
perform equally well for large numbers of particles, and that (w)PB outperform CBO for few particles.
This could be explained by the fact that the probability of all particles deviating from the global minimum decreases as their number increases.

\begin{table}[htp]
\centering
\begin{minipage}{0.45\textwidth}
\begin{filecontents*}{toy2d.csv}
Dim,Nparticles,CBO,PB,WPB
2,4,0.6572,0.8212,0.7596
2,8,0.7168,0.78,0.7684
2,16,0.7768,0.7684,0.7736
\end{filecontents*}
\csvreader[myTableStyle]{toy2d.csv}{}%
{\Dim&\Nparticles~& \CBO &\PB & \WPB}
\end{minipage}
\caption{Success rates of consensus based optimization (CBO), personal best (PB) and weighted personal best (wPB) scheme for the 2D tow problem in space dimension $2$ for different $\#$ of particles.}
\label{t:toy2d}
\end{table}

\section{Conclusion} In this paper we introduced a consensus based global optimization scheme, which includes the personal best information of each particle. The proposed 
generalization is motivated by the original works on particle swarm algorithms, in which particles adjust their position as a linear combination of moving towards the current global
best and  their personal best value. 

We discussed how information about the personal best can be included in consensus based optimization schemes, leading to a system of functional stochastic differential 
equations. A well-posedness result for the respective regularized non-Markovian SDEs was presented. New features of the algorithm with personal best were illustrated and compared in computational experiments. The numerical results indicate that information about the personal best leads to higher success rates in the case of few particles and that the corresponding weighted means are better approximations of the global function minima.\\

\section*{Acknowledgments}
The authors would like to thank Christoph Belak (TU Berlin), Stefan Grosskinsky (University of Warwick), Greg Pavliotis (Imperial College London) and Oliver Tse (TU Eindhoven) for the helpful discussions and constructive input.
CT was partly supported by the European Social Fund and by the Ministry Of Science, Research and the Arts Baden-W\"urttemberg. MTW was partly supported by the New Frontier's grant NST-0001 of the Austrian Academy of Sciences \"OAW.

\section*{Appendix}
\begin{proof}{[Theorem \ref{t:wellposed}]}
	In the following we denote in abuse of notation by $v_f[X]$ the vector $(v_f, \dots, v_f) \in \mathbb R^{dN}.$ We rewrite \eqref{eq:particlePSO} with $\lambda(X^i(t),v_f,p_f^i) \equiv \lambda, \mu(X^i(t),v_f,p_f^i) \equiv 0$ as 
	\[
	dX(t) = -\lambda (X(t) - v_f[X(t)]) dt + \sqrt{2}\sigma \text{diag}(X(t) - v_f[X(t)]) dB_t  
	\]
	with $\text{diag}(X(t) - v_f[X(t)]) \in \mathbb R^{dN \times dN}$ being the diagonal matrix with $d_k = diag(X^k(t) - v_f) \in \mathbb R^{d \times d} $ for $k=1, \dots, N$ and $dB_t$ a $dN$-dimensional Brownian motion. 
	The argument follows the lines of the well-posedness in \cite{CBO2}. In fact, let $M[X(t)] = \text{diag}(X(t) - v_f[X(t)])$ and $n\in \mathbb N$ arbitrary. We have to check that there exists a constant $C_n$ such that
	\begin{equation}\label{eq:inequ}
	-2\lambda X(t) \cdot (X(t) - v_f[X(t)]) + 2\sigma^2 \text{trace}(M[X(t)] M[X(t)]^T) \le C_n |X(t)|^2,
	\end{equation}
	for every $|X(t)| \le n.$ Note that $f(X(t))$ is bounded for $|X(t)| \le n$ due to its local Lipschitz continuity. Hence,
	the estimate for the first term on the left-hand side is identical to the one in \cite{CBO2}. Indeed, we have
	\[
	-2\lambda X(t) \cdot (X(t) - v_f[X(t)]) \le 2\lambda \sqrt{N} |X(t)|^2.
	\]
	For the component-wise drift we obtain
	\[
	2\sigma^2\text{trace}\left(M [X(t)] M[X(t)]^T\right)  = 2\sigma^2\sum_{j=1}^N \sum_{k=1}^d [(X^{j}(t) - v_f)_k]^2 \le 4\sigma^2 (1 + N) |X(t)|^2.
	\]
	Combining the two preceding estimates we obtain Eq. \eqref{eq:inequ}. Now, employing \cite[Ch 5.3,Thm 3.2]{Durrett} yields the desired result.
\end{proof}

\begin{ }{[Lemma \ref{lem:estimatesOnPV}]} We start by showing continuity. For $p_f[\varphi]$ we need to check continuity as $t \rightarrow 0$. In fact, l'Hospital's rule yields
	\[
	\lim\limits_{t\rightarrow 0} \frac{\int_0^t \varphi^i(s) e^{-\beta f(\varphi^i(s))} ds}{\int_0^t e^{-\beta f(\varphi^i(s))} ds} = \lim\limits_{t\rightarrow 0} \frac{\varphi^i(t) e^{-\beta f(\varphi^i(t))}}{e^{-\beta f(\varphi^i(t))} } =  \lim\limits_{t\rightarrow 0} \varphi^i(t) = \varphi_0^i.
	\]
	This directly implies the continuity of the vector $p_f[\varphi].$ Moreover, it is easy to see that
	\begin{align*}
	|v_f[\varphi(t)]|^2 &= \left|  \frac{\sum_{i=1}^N \varphi^i(t) e^{-\alpha f(\varphi^i(t))}}{\sum_{i=1}^N e^{-\alpha f(\varphi^i(t))}} \right|^2 \le |\varphi(t)|^2, \\ 
	|p_f[\varphi](t)|^2 &= \sum_{i=1}^N |p_f^i[\varphi](t)|^2 =  \sum_{i=1}^N \left| \frac{\int_0^t \varphi^i(s) e^{-\beta f(\varphi^i(s))} ds}{\int_0^t e^{-\beta f(\varphi^i(s))} ds} \right|^2 \le \|\varphi\|_t^2.
	\end{align*}
	We are left to show the local Lipschitz continuity. Therefore, we estimate
	\begin{align*}
	\left|p_f^i[\varphi](t) - p_f^i[\hat \varphi](t) \right| &\le \left| \frac{\int_0^t (\varphi^i(s) - \hat \varphi^i(s)) e^{-\beta f(\varphi^i(s)) } ds}{\int_0^t e^{-\beta f(\varphi^i(s))} ds} \right| + \left| \frac{\int_0^t \hat \varphi^i(s) \left( e^{-\beta f(\varphi^i(s))} - e^{-\beta f(\hat \varphi^i(s))} \right) ds}{\int_0^t e^{-\beta f(\varphi^i(s))} ds}\right| \\ &\quad + \left| \frac{\int_0^t \hat \varphi^i(s) e^{-\beta f(\hat \varphi^i(s))} ds \left( \int_0^t \hat \varphi^i(s) e^{-\beta f(\hat \varphi^i(s))} ds - \int_0^t \varphi^i(s) e^{-\beta f(\varphi^i(s))} ds \right)}{\int_0^t e^{-\beta f(\varphi^i(s))} ds \, \int_0^t e^{-\beta f(\hat \varphi^i(s))} ds} \right| \\
	&=: I_1 + I_2 + I_3,
	\end{align*}
	to obtain
	\begin{align*}
	I_1 &\le \|\varphi^i - \hat \varphi^i\|_t, ~~ I_2 \le \beta e^{ \beta (\overline{f}- \underline{f})} L_f n\, \|\varphi^i - \hat \varphi^i\|_t, \text{ and } I_3 \le  \beta e^{ \beta (\overline{f}- \underline{f})} (1 + L_f n) \, \|\hat \varphi^i\|_t\, \|\hat \varphi^i - \varphi^i\|_t, 
	\end{align*}
	with $L_f$ being the global Lipschitz constant of $f$ on $B_n = \{ x \in R^d \colon |x| \le n\}$ and $\underline f, \overline f$ are the minimal and maximal value of $f$ on $B_n,$ respectively.
	Altogether, this yields the estimate
	\[
	\left|p_f^i[\varphi](t) - p_f^i[\hat \varphi](t) \right| \le \Big( 1+ (1 + 2L_f)   \beta n e^{ \beta (\overline{f}- \underline{f})}  \Big)\, \|\hat \varphi^i - \varphi^i\|_t.
	\]
	For the vectors $p_f(t) = (p_f^i[\varphi](t))_{i=1,\dots,N}$ and $\hat p_f(t) = (\hat p_f^i[\varphi](t))_{i=1,\dots,N}$ this implies
	\[
	|p_f(t) - \hat p_f(t)|^2 = \sum_{i=1}^N \left|p_f^i[\varphi](t) - p_f^i[\hat \varphi](t) \right|^2 \le \Big( 1+ (1 + 2L_f)   \beta n e^{ \beta (\overline{f}- \underline{f})}  \Big)\|\varphi - \hat \varphi\|_t^2.
	\]
	Similarly, we have
	\begin{align*}
	\left|v_f[\varphi(t)] - v_f[\hat \varphi(t)] \right| &\le \left| \frac{\sum_{i=1}^N \big(\varphi^i(t) - \hat \varphi^i(t)\big) e^{-\alpha f(\varphi^i(t))}}{\sum_{i=1}^N  e^{-\alpha f(\varphi^i(t))}} \right| + \left| \frac{\sum_{i=1}^N \hat \varphi^i(t) \left( e^{-\alpha f(\varphi^i(t))} - e^{-\alpha f(\hat \varphi^i(t))} \right)}{\sum_{i=1}^N  e^{-\alpha f(\varphi^i(t))}} \right| \\
	&\qquad+ \left| \frac{\sum_{i=1}^N \hat \varphi^i(t) e^{-\alpha f(\hat \varphi^i(t))} \left( \sum_{i=1}^N \hat \varphi^i(t) e^{-\alpha f(\hat \varphi^i(t))} - \sum_{i=1}^N \varphi^i(t) e^{-\alpha f(\varphi^i(t))} \right)}{\left( {\sum_{i=1}^N  e^{-\alpha f(\varphi^i(t))}}  \right) \left( {\sum_{i=1}^N  e^{-\alpha f(\hat \varphi^i(t))}} \right)} \right| \\
	&=: J_1 + J_2 + J_3,
	\end{align*}
	which satisfy
	\begin{align*}
	J_1 &\le |\varphi(t) - \hat \varphi(t)|_1, \quad J_2 \le \frac{\alpha n L_f e^{-\alpha \underline{f}}}{N}|\varphi(t) - \hat \varphi(t)|_1, \text{ and } J_3 \le  n e^{\alpha(\overline{f} - \underline{f})} \left( \frac{1}{N} + \alpha n L_f  \right) |\hat \varphi(t)- \varphi(t)|_1.
	\end{align*}
	Thus, we get
	\[
	\left|v_f[\varphi(t)] - v_f[\hat \varphi(t)] \right| \le \left( 1+ \frac{\alpha n L_f e^{-\alpha \underline{f}}}{N} + n e^{\alpha(\overline{f} - \underline{f})} \left( \frac{1}{N} + \alpha n L_f  \right) \right) |\hat \varphi(t) - \varphi(t)|_1.
	\]
	Taking squares leads to the estimate
	\[
	\left|v_f[\varphi(t)] - v_f[\hat \varphi(t)] \right|^2 \le \left( 1+ \frac{\alpha n L_f e^{-\alpha \underline{f}}}{N} + n e^{\alpha(\overline{f} - \underline{f})} \left( \frac{1}{N} + \alpha n L_f  \right) \right)^2 2^{N-1} |\hat \varphi(t) - \varphi(t)|^2.
	\]
\end{ }

\end{document}